\documentclass[12pt, leqno]{article}

\usepackage{amsmath,amssymb}

\RequirePackage{fix-cm}

\usepackage{manfnt}
\usepackage{palatino}
\usepackage[sc]{mathpazo}
\usepackage{eulervm}
\usepackage{fullpage}

\newtheorem{theorem}{Theorem}[section]

\date{}

\newcommand{\R}{\mathbb R}
\newcommand{\N}{\mathbb N}

\renewcommand{\d}{\mathrm d}
\DeclareMathOperator{\rang}{rang}
\newcommand{\bra}[1]{\left( #1 \right)}

\title{\bf Homeomorphism criteria\\for the theory of grid generation}
\author{Marina\,F.\,Prokhorova
\footnote{The work was partially supported by the RFBR grants 09-01-00173-a and 09-01-00139-a (Russia),
by the project of oriented fundamental research of the Ural Branch of Russian Academy of Sciences, 
and by the Program for Basic Research of Mathematical Sciences Branch of Russian Academy of Sciences.
} }

\begin{document}
\maketitle

\begin{abstract}
We give some general criteria of being a homeomorphism for continuous mappings of topological manifolds,
as well as criteria of being a diffeomorphism for smooth mappings of smooth manifolds.
As an illustration, we apply these criteria to the problems arising in two- and three-dimensional grid generation.

\end{abstract}

\section*{Introduction}

While developing algorithms of grid generation for numerical calculations in
domains of complicated configuration, it is necessary to apply (see, for
example, \cite{Ushakova1}, \cite{Ushakova2}) different criteria for finding out
whether a continuous mapping is a homeomorphism or a smooth mapping is a
diffeomorphism.
Such criteria for bounded domains in~$\R^n$ were suggested
in~\cite{Bobylev1, Bobylev2}. Unfortunately, some theorems in~\cite{Bobylev1,
Bobylev2} are incorrect, while the proofs of some others are incomplete 
(see the last section of \cite{Prokh} for details).

In this paper we formulate and prove some general criteria of being a
homeomorphism for continuous mappings of topological manifolds (Section 1),
as well as criteria of being a diffeomorphism for smooth mappings of smooth manifolds (Section 2).
In Section 3 we give some applications of these criteria to the problems arising in two- and three-dimensional grid generation.

For the convenience of a reader who deals not with topology but with numerical
methods, in the paper we give definitions of main topological notions mentioned
in theorem formulations. It is possible to consider them in more detail, for
example, in~\cite{R-F}, \cite{F-F}, \cite{Dold}.

\section{Topological Manifolds}
\label{sec:top_man}

Let $X$ and $Y$ be topological spaces.
A (continuous) mapping $f\colon X\to Y$ is called an \textit{immersion}
if any point $x \in X$ has a neighborhood $U$ such that $\left.f \right|_{U}$
is a homeomorphism of $U$ onto $f(U)$.

A Hausdorff topological space $X$
is called an \textit{$n$-dimensional topological manifold} if
there exist countable cover of $X$ by open sets each of that
is homeomorphic either to the space $\R ^n$ or to the half-space $\R^n_+ = \{(x_1,\ldots,x_n) \colon x_i \in \R, x_1 \geq 0\}$.

Points $x\in X$ having a neighborhood homeomorphic to $\R^n$ are called
\textit{interior points}. The subspace of $X$ consisting of points that are not
interior is called the \textit{boundary} $\partial X$ of the manifold $X$.
The boundary of an $n$-dimensional manifold is an $(n-1)$-dimensional manifold
without boundary (that is $\partial\partial X = \emptyset$).

We will based on Theorem 8 from \cite{Prokh}.
For the compact manifold $X$ it takes the following form:

\begin{theorem}\label{t1}
Suppose that $X$ and $Y$ are connected topological manifolds of equal dimensions,
$X$ is compact,
$f \colon X \to Y$ is a continuous mapping,
$f(\partial X) \subseteq \partial Y$,
$\left.f \right|_{\partial X}$ and $\left.f \right|_{X \backslash \partial X}$ are immersions.
Then $f$ is a finite-fold covering.
\end{theorem}

\noindent
\textbf{Remarks.}
\vspace{-5pt}
\begin{itemize}
\itemsep=-2pt
    \item
    The condition `$\left.f \right|_{\partial X}$ is an immersion' means that
any point $x \in \partial X$ has \textit{a neighborhood $U$ in $\partial X$} such that $\left.f \right|_{U}$
is a homeomorphism of $U$ onto $f(U)$.
So $U$ is open in $\partial X$ but it is \textit{not} open in $X$.
    \item The condition `$\left.f \right|_{\partial X}$ and $\left.f \right|_{X \backslash \partial X}$ are
immersions' is weaker then the condition `$f$ is an immersion'.
    \item Remember that a continuous mapping $f \colon X \to Y$ is called a covering if any point
$y\in Y$ has a neighborhood $V$ such that $f$ maps every connected component of
$f^{-1}(V)$ onto $V$ homeomorphically.
    \item For a covering $f \colon X \to Y$ with connected $Y$ the cardinality of $f^{-1}(y)$ is the same for all points $y\in Y$;
    if it is equal $m$, $m\in \N$, then $f$ is called an $m$-fold covering. 
		In our case $f$ is finite-fold due to compactness of $X$.
\end{itemize}

Theorems \ref{t2}-\ref{t7} are straightforward corollaries of this theorem.

\begin{theorem}\label{t2}
Suppose that $X$, $Y$ and $f \colon X \to Y$ are satisfied to the conditions of Theorem
\ref{t1}, and there exists a point in $Y$ with one-point preimage.
Then $f$ is a homeomorphism of $X$ onto $Y$.
\end{theorem}

\begin{theorem}\label{t3}
Suppose that $X$ and $Y$ are connected topological manifolds of equal dimensions with nonempty boundaries,
$X$ is compact,
$f \colon X \to Y$ is a continuous mapping
that injectively maps $\partial X$ to $\partial Y$,
and $\left.f \right|_{X \backslash\partial X}$ is an immersion.
Then $f$ is a homeomorphism of $X$ onto $Y$.
\end{theorem}

\begin{proof}
$\left.f \right|_{\partial X}$ considering as a mapping from $\partial X$ to $f(\partial
X)$ is a continuous bijective mapping of compact topological spaces,
so it is a homeomorphism.
For any internal point $x$ of $X$ its image lie in some subset of $Y$
homeomorphic to $\R ^n$ so $f(x)$ cannot lie in $\partial Y$.
Thus any point $y \in f(\partial X)$ has one-point preimage,
$\left.f \right|_{X \backslash\partial X}$ and $\left.f \right|_{\partial
X}$ are immersions,
and $f$ is a homeomorphism of $X$ onto $Y$ by Theorem \ref{t1}.
\end{proof}

\begin{theorem}\label{t4}
Suppose that $X$, $Y$ and $f \colon X \to Y$ are satisfied to the
conditions of Theorem \ref{t1}, and there is no proper subgroup of
finite index in $\pi_1(Y)$ isomorphic to $\pi_1(X)$. Then $f$ is a
homeomorphism of $X$ onto $Y$.
\end{theorem}

\noindent
\textbf{Remarks.}
\vspace{-5pt}
\begin{itemize}
\itemsep=-2pt
    \item If $Y$ is simply connected then
$\pi_1(Y)={1}$ and has no proper subgroups so the last condition of
the theorem is fulfilled.
    \item Remember that the fundamental group $\pi_1(X)$ of a topological space $X$
is the group whose elements are equivalence classes of the loops
in $X$ with some fixed based point; two loops are considered
equivalent if there is a homotopy from one loop to another. 
A space is said to be simply connected if its fundamental group is
trivial (so any two loops are homotopic). 
The proper subgroup is a subgroup which is not coincide with the whole group. 
The index of a subgroup $H$ in a group $G$ is the number of cosets of $H$ in $G$ (in
particular, proper subgroup is a subgroup of index greater then
1).
\end{itemize}

\begin{proof}
By Theorem \ref{t1}, $f$ is a covering so it induces injective homomorphism $\pi_1(X) \to \pi_1(Y)$. 
This covering is finite-fold so the image of $\pi_1(X)$ in
$\pi_1(Y)$ is the subgroup of finite index. If no proper subgroup
of finite index in $\pi_1(Y)$ is isomorphic to $\pi_1(X)$ then
this homomorphism is an isomorphism. So $f$ is a homeomorphism of
$X$ onto $Y$.
\end{proof}

\begin{theorem}\label{t4b}
Suppose that $X$, $Y$ and $f \colon X \to Y$ are satisfied to the
conditions of Theorem \ref{t1}.
Let the Euler characteristics of $X$, $Y$ are equal and non-zero:
$\chi(X)=\chi(Y)\neq 0$. 
Then $f$ is a homeomorphism of $X$ onto $Y$.
\end{theorem}

Remember that the Euler characteristic of a topological space $X$ is the alternating sum of its Betti numbers 
$\chi(X) = \sum_{k=0}^{\infty}{(-1)^k b_k(X)}$ in the case when this sum is well-defined. 
If $X$ is a compact topological manifold then this sum is well-defined and equal to the alternating sum of the numbers of $k$-dimensional cells in a cell decomposition of $X$. 

\begin{proof}
By Theorem \ref{t1}, $f$ is an $m$-fold covering for some natural number $m$.
Therefore $\chi(X) = m \chi(Y)$,
and taking account of the condition of the Theorem we obtain $m=1$.
Thus $f$ is an $1$-fold covering that is a homeomorphism of $X$ onto $Y$.
\end{proof}

\begin{theorem}\label{t5}
Suppose that $X$, $Y$ and $f \colon X \to Y$ are satisfied to the
conditions of Theorem \ref{t1}, and $X$, $Y$ have nonempty
boundaries with the same number of the connected components:
$\partial X = \amalg_{i=1}^k A_i$, $\partial Y = \amalg_{i=1}^{k}
B_i$. Suppose that there exist $j$ such that $f(A_j) \subseteq
B_j$ and there is no proper subgroup of finite index in
$\pi_1(B_j)$ isomorphic to $\pi_1(A_j)$. Then $f$ is a
homeomorphism of $X$ onto $Y$.
\end{theorem}

\begin{proof}
For any internal point $x$ of $X$ its image lie in some subset of $Y$
homeomorphic to $\R ^n$ so $f(x)$ cannot lie in $\partial Y$.
Therefore $f^{-1}(\partial Y) = \partial X$.
For any $i$ $A_i$ is connected so $f(A_i)$ is contained in some $B_{i'}$.
From this and the surjectivity of $f$ we obtain that
the preimage of every connected component of $\partial Y$ is the one
connected component of $\partial X$.
In particular, $f^{-1}(B_j) = A_j$.

By Theorem \ref{t1}, $f$ is an $m$-fold covering, so $\left.f \right|_{A_j}$
considered as a mapping from $A_j$ to $B_j$ is also an $m$-fold covering. 
If no proper subgroup of finite index in $\pi_1(B_j)$ is
isomorphic to $\pi_1(A_j)$ then $\left.f \right|_{A_j}$ is a
homeomorphism of $A_j$ onto $B_j$, $m=1$, and $f$ is a homeomorphism of $X$ onto $Y$.
\end{proof}

The following result is the immediate consequence of Theorem \ref{t5}:

\begin{theorem}\label{t6}
Suppose that $X$, $Y$ and $f \colon X \to Y$ are satisfied the conditions of Theorem
\ref{t1},
$X$, $Y$ have nonempty boundaries with the same number of the connected
components,
and at least one of the connected components of $\partial Y$ is simply
connected.
Then $f$ is a homeomorphism of $X$ onto $Y$.
\end{theorem}

\begin{theorem}\label{t7}
Suppose that $X$, $Y$ are compact connected topological manifolds,
their boundaries are nonempty and have the same number of the connected
components:
$\partial X = \amalg_{i=1}^k A_i$,
$\partial Y = \amalg_{i=1}^{k} B_i$,
where the components are numerated in the order of non-increasing of
their Euler characteristics:
$\chi(A_1)\geq\ldots\geq\chi(A_k)$,
$\chi(B_1)\geq\ldots\geq\chi(B_k)$.
Suppose that
\begin{equation}\label{eq_chi}
\mbox{there is no natural } m\geq 2 \mbox{ such that }
\chi(A_i) = m \chi(B_i) \mbox{ for all } i=1\ldots k.
\end{equation}
At this conditions, if $f \colon X \to Y$ is a continuous mapping such that
$f(\partial X) \subseteq \partial Y$ and
$\left.f \right|_{\partial X}$, $\left.f \right|_{X \backslash \partial X}$ are
immersions,
then $f$ is a homeomorphism of $X$ onto $Y$.
\end{theorem}

\smallskip

\begin{proof}
By Theorem \ref{t1}, $f$ is an $m$-fold covering for some natural number $m$.
As in the proof of Theorem \ref{t5}, $f^{-1}(\partial Y) = \partial X$,
and $\left.f \right|_{A_i}$ considered as a mapping from $A_i$ to $f \bra{A_i}$ is also an $m$-fold covering.
Therefore $\chi\bra{A_i} = m \chi \bra {f \bra{A_i}}$,
and $f(A_i)=B_i$ due to decreasing order of
$\chi\bra{A_i}$ and $\chi \bra {B_i}$.
So $\chi(A_i) = m \chi(B_i)$,
and taking account of the condition of the Theorem we obtain $m=1$.
Thus $f$ is an $1$-fold covering that is a homeomorphism of $X$ onto $Y$.
\end{proof}

In the previous theorems, it was required that the boundary of $X$ be mapped to the boundary
of $Y$.
Theorem 10 from \cite{Prokh} is free of this condition;
for the compact manifold $X$ it takes the following form:

\begin{theorem}\label{ti}
Suppose that $X$ and $Y$ are connected topological manifolds of equal dimensions,
$X$ is compact, $Y$ is simply connected,
$f \colon X \to Y$ is an immersion,
and the restriction of $f$ to every connected component of $\partial X$ is an injection.
Then $f$ is a homeomorphism of $X$ onto $f(X)$.
\end{theorem}

\section{Smooth Manifolds}

Let $X$ be an $n$-dimensional topological manifold.
A {\it chart} of $X$ is a homeomorphism $\varphi$ of an open domain $U \subset X$
to $\R^n$ or to $\R^n_+$. Charts $\varphi \colon U \to V$ and $\varphi' \colon
U' \to V'$ are called $C^r$-compatible if the mapping $\varphi' \circ
\varphi^{-1}$ and the inverse mapping are $C^r$-smooth in their domains. A {\it
$C^r$-structure} on $X$ is a cover of $X$ by pairwise $C^r$-compatible charts.
Two $C^r$-structures on $X$ are considered equivalent if their union is also a
$C^r$-structure.

A {\it $C^r$-manifold} (a manifold of smoothness class $C^r$) is an
$n$-dimensional topological manifold $X$ with an equivalence class of $C^r$-structures on it.

A continuous mapping $f\colon X \to X'$, where $X$ and $X'$ are $C^r$-manifolds, is called
a smooth mapping of class $C^r$, or {\it $C^r$-mapping}, if for any pair of
charts $\varphi\colon U\to V$ and $\varphi'\colon U'\to V'$ ($U \subset X$,
$U'\subset X'$) the mapping $\varphi'\circ f\circ\varphi^{-1}$ is $C^r$-smooth
in its domain.

A {\it $C^r$-diffeomorphism} is a bijective $C^r$-mapping such that its inverse is also smooth mapping of class $C^r$.

The results of the previous section can be easily reformulated for the
smooth case. Let us present the smooth variants of them.
Everywhere below, $r\geq 1$.

\begin{theorem}\label{ts1}
Suppose that $X$ and $Y$ are connected $C^r$-manifolds of equal dimensions,
$X$ is compact,
$f \colon X \to Y$ is a $C^r$-mapping, whose differential $\d f$ is nondegenerate (that is $\rang (\d f) = \dim X$) on $X$,
and $f(\partial X) \subseteq \partial Y$.
Then $f$ is a smooth covering.
\end{theorem}

\begin{theorem}\label{ts2}
Suppose that $X$, $Y$ and $f \colon X \to Y$ are satisfied to the conditions of Theorem
\ref{ts1}, and there exists a point in $Y$ with one-point preimage.
Then $f$ is a $C^r$-diffeomorphism of $X$ onto $Y$.
\end{theorem}

\begin{theorem}\label{ts3}
Suppose that $X$ and $Y$ are connected $C^r$-manifolds of equal dimensions with nonempty boundaries,
$X$ is compact,
$f \colon X \to Y$ is a $C^r$-mapping,
whose differential $\d f$ is nondegenerate on $X \backslash\partial X$,
and $f$ injectively maps $\partial X$ to $\partial Y$.
Then $f$ is a homeomorphism of $X$ onto $Y$,
and $\left.f \right|_{X \backslash\partial X}$ is a $C^r$-diffeomorphism of $X \backslash\partial X$ onto $Y \backslash\partial Y$.
If differential $\d f$ is nondegenerate on $X$ then
$f$ is a $C^r$-diffeomorphism of $X$ onto $Y$.
\end{theorem}

\begin{theorem}\label{ts4}
Suppose that $X$, $Y$ and $f \colon X \to Y$ are satisfied to the
conditions of Theorem \ref{ts1}, and there is no proper subgroup
of finite index in $\pi_1(Y)$ isomorphic to $\pi_1(X)$. Then $f$
is a $C^r$-diffeomorphism of $X$ onto $Y$.
\end{theorem}

\begin{theorem}\label{ts4b}
Suppose that $X$, $Y$ and $f \colon X \to Y$ are satisfied to the
conditions of Theorem \ref{ts1}.
Let the Euler characteristics of $X$, $Y$ are equal and non-zero:
$\chi(X)=\chi(Y)\neq 0$. 
Then $f$ is a $C^r$-diffeomorphism of $X$ onto $Y$.
\end{theorem}

\begin{theorem}\label{ts5}
Suppose that $X$, $Y$ and $f \colon X \to Y$ are satisfied to the
conditions of Theorem \ref{ts1}, and $X$, $Y$ have nonempty
boundaries with the same number of the connected components:
$\partial X = \amalg_{i=1}^k A_i$, $\partial Y = \amalg_{i=1}^{k}
B_i$. Suppose that there exist $j$ such that $f(A_j) \subseteq
B_j$ and there is no proper subgroup of finite index in
$\pi_1(B_j)$ isomorphic to $\pi_1(A_j)$. Then $f$ is a
$C^r$-diffeomorphism of $X$ onto $Y$.
\end{theorem}

\begin{theorem}\label{ts6}
Suppose that $X$, $Y$ and $f \colon X \to Y$ are satisfied to the conditions of Theorem
\ref{ts1},
$X$, $Y$ have nonempty boundaries with the same number of the connected
components,
and at least one of the connected components of $\partial Y$ is simply
connected.
Then $f$ is a $C^r$-diffeomorphism of $X$ onto $Y$.
\end{theorem}

\begin{theorem}\label{ts7}
Suppose that $X$, $Y$ are three-dimensional compact connected $C^r$-manifolds,
their boundaries are nonempty and have the same number of the connected
components:
$\partial X = \amalg_{i=1}^k A_i$,
$\partial Y = \amalg_{i=1}^{k} B_i$,
where the components are numerated in the order of non-increasing of
their Euler characteristics.
Suppose that condition (\ref{eq_chi}) is fulfilled,
$f \colon X \to Y$ is a $C^r$-mapping, whose differential $\d f$ is nondegenerate on $X$,
and $f(\partial X) \subseteq \partial Y$.
Then $f$ is a $C^r$-diffeomorphism of $X$ onto $Y$.
\end{theorem}

See the description of some cases when condition (\ref{eq_chi}) is fulfilled
in the remark to Theorem \ref{t7}.

\begin{theorem}\label{tsi}
Suppose that $X$ and $Y$ are connected $C^r$-manifolds of equal dimensions,
$X$ is compact, $Y$ is simply connected,
$f \in C^r(X,Y)$  is a $C^r$-mapping, whose differential $\d f$ is nondegenerate on $X$,
and the restriction of $f$ to every connected component of $\partial X$ is an injection.
Then $f$ is a $C^r$-diffeomorphism of $X$ onto $f(X)$.
\end{theorem}

\section{Some applications}

Here we give some remarks on the applications of the criteria from previous two sections to the problems arising in grid generation.

Note that to apply the criteria from Theorems 1, 4, 5, 6, 7 and 8 of every section
we have to verify only local conditions on the mapping $f$.
There is only one global condition on the mapping $f$ in Theorems 2, 3 and 9:
in Theorem 2 it deals with the preimage of one point;
in Theorems 3 and 9 this global condition deal with the restriction of $f$ to the boundary of the manifold,
while the restriction of $f$ to the interior of the manifold has to satisfy only local condition.

Theorems \ref{ti} and \ref{tsi} can be used in the case when compact $n$-dimensional manifold $X$ is mapped into $\R^n$.

Two- and three-dimensional grids are very often used in numerical calculations.
So we describe the using of our results for these cases more thoroughly.

In \textbf{two-dimensional case} one can use Theorems \ref{t4b} and \ref{ts4b}. 
Remember that the genus $g$ of a closed orientable surface is the number of the handles that we have to attach to the sphere to obtain
the surface. 
Similarly, the genus $g$ of a compact orientable surface with nonempty boundary is the number of handles that we have to attach to the disk or to the disk with holes to obtain the surface. 

The Euler characteristic of a closed orientable surface is $\chi(X) = 2-2g$.
If $X$ is a compact orientable surface that is a sphere with $g$ handles and $b$ holes then $\chi(X) = 2-2g-b$.
The Euler characteristic of a compact non-orientable surface is half its $2$-fold otientable covering. 

So if the surfaces $X$ and $Y$ are compact homeomorphic surfaces distinguishing from 
the annulus, torus, M\"obius band, and Klein bottle, then
Theorems \ref{t4b} and \ref{ts4b} provide us the condition for $f$ to be a homeomorphism or a diffeomorphism correspondingly. 

In \textbf{three-dimensional case} Theorems \ref{t7} and \ref{ts7} may be useful. 
Note that there is simple particular case when condition (\ref{eq_chi}) is fulfilled, 
namely the case when at least one of the connected components of $\partial Y$ is
homeomorphic to a sphere (Theorems \ref{t6}, \ref{ts6} describes this case as well).

Also in three-dimensional case one can use the following variants of Theorems \ref{t7}, \ref{ts7}:

\begin{theorem}\label{t8}
Suppose that $X$, $Y$ are three-dimensional compact connected topological manifolds,
their boundaries are nonempty and have the same number of the connected
components, at least one of which is not homeomorphic neither torus nor Klein bottle.
Then we have two alternatives:
\vspace{-5pt}
\begin{enumerate}
\itemsep=-2pt
    \item
    either the sets $\bra{ \chi \bra{A_i}}$ and $\bra{ \chi \bra{B_i}}$ are coincide (taking account of multiplicity),
    \item
    or these sets are different.
\end{enumerate}
\vspace{-5pt}
At the first case any mapping $f \colon X \to Y$ satisfying the conditions
of Theorem \ref{t1} is a homeomorphism of $X$ onto $Y$.
At the second case there is no homeomorphism of $X$ onto $Y$ at all.
\end{theorem}

\begin{theorem}\label{ts8}
Suppose that $X$, $Y$ are three-dimensional compact connected $C^r$-manifolds,
their boundaries are nonempty and have the same number of the connected
components, at least one of which is not homeomorphic neither
torus nor Klein bottle.
Then we have two alternatives:
\vspace{-5pt}
\begin{enumerate}
\itemsep=-2pt
    \item
    either the sets $\bra{ \chi \bra{A_i}}$ and
    $\bra{ \chi \bra{B_i}}$ are coincide (taking account of multiplicity),
    \item
    or these sets are different.
\end{enumerate}
\vspace{-5pt}
At the first case any mapping $f \colon X \to Y$ satisfying the conditions
of Theorem \ref{ts1} is a $C^r$-diffeomorphism of $X$ onto $Y$.
At the second case there is no homeomorphism of $X$ onto $Y$ at all.
\end{theorem}

\begin{proof}
The second case is obvious.
To prove the first case it is sufficient to use Theorem \ref{t7} 
and the fact that a closed surface with zero Euler characteristic is either torus or Klein bottle.
\end{proof}

Note that all these results are applicable particularly when $X$, $Y$ are embedded into $\R^3$,
that is $X$, $Y$ are closures of bounded open regions in $\R^3$ bounded by locally flat surfaces.
In this case Theorems \ref{t8}, \ref{ts8} are applicable with the sole exception: 
when $X$, $Y$ are the solid tori with the small solid tori removed.
For this exceptional case one can use Theorems \ref{t2}, \ref{t3} and their smooth analogues \ref{ts2}, \ref{ts3}. 
In addition, one can use Theorems \ref{t4}, \ref{ts4} if the inclusions of removing solid tori into the big solid torus are non-trivial.

\bigskip
\bigskip
\noindent
IMM UrBr RAS \& Ural Federal University \\
S.Kovalevskaya 16, Ekaterinburg 620990 Russia \\
e-mail: pmf@imm.uran.ru

\end{document}